\documentclass[11pt]{amsart}
\usepackage{amsfonts}
\usepackage{epsf,graphicx}
\usepackage{amscd}
\usepackage{amsmath,latexsym,amssymb,amsthm}
\usepackage{hyperref}
\usepackage{setspace,cite}
\usepackage{lscape,fancyhdr,fancybox}
\usepackage{a4}
\usepackage{dsfont,texdraw}
\usepackage{mathtools}

\def\Kmath{\mathbb{K}}

\newcounter{cnt1}
\newcounter{cnt2}
\newcounter{cnt3}
\newcommand{\blr}{\begin{list}{$($\roman{cnt1}$)$} {\usecounter{cnt1}
        \setlength{\topsep}{0pt} \setlength{\itemsep}{0pt}}}
\newcommand{\bla}{\begin{list}{$($\alph{cnt2}$)$} {\usecounter{cnt2}
        \setlength{\topsep}{0pt} \setlength{\itemsep}{0pt}}}
\newcommand{\bln}{\begin{list}{$($\arabic{cnt3}$)$} {\usecounter{cnt3}
                \setlength{\topsep}{0pt} \setlength{\itemsep}{0pt}}}
\newcommand{\el}{\end{list}}

\newtheorem{Thm}{Theorem}[section]
\newtheorem{Lem}[Thm]{Lemma}
\newtheorem{Prop}[Thm]{Proposition}
\newtheorem{Def}[Thm]{Definition}
\newtheorem{Exm}[Thm]{Example}
\newtheorem{Rem}[Thm]{Remark}
\newtheorem{Cor}[Thm]{Corollary}
\newtheorem{Note}[Thm]{Note}

\begin{document}

\title{Cohomology and deformations of Courant pairs}
\author{Ashis Mandal and Satyendra Kumar Mishra}
\footnote{The research of S.K. Mishra is supported by CSIR-SPM fellowship grant 2013 .}

\begin{abstract}{
In this note we define a notion of Courant pair as a Courant algebra over the Lie algebra of linear derivations on an associative algebra. We study formal deformations of Courant pairs by constructing a  cohomology bicomplex with coefficients in a module from the cochain complexes defining Hochschild cohomology and Leibniz cohomology.

}
\end{abstract}
%\date{ }
%\newsymbol \rtimes 226F
\maketitle
\section{Introduction}
The aim of this article is to define a notion of Courant pair  as a Courant algebra over the Lie algebra $Der(A)$ of linear derivations on an associative algebra  $A$ and describe a one parameter formal deformation theory of Courant pairs. We construct a cohomology bicomplex using the Hochschild cohomology for associative algebras and cohomology of Leibniz algebras respectively. 

The notion of Courant algebras was defined by H. Bursztyn and coauthors ( \cite{BCG} ) in the context of reduction of Courant algebroids and generalized complex structures. It is used to interpret the moment map in symplectic geometry as an object which controls an extended part of the action of Courant algebras on Courant algebroids. By definition, a Courant algebra over a Lie algebra $\mathfrak{g}$  is a Leibniz algebra $L$ equipped with a Leibniz algebra homomorphism to the Lie algebra $\mathfrak{g}$. One special case of this type of algebras is an exact Courant algebra where the Leibniz algebra homomorphism is a surjective map onto $\mathfrak{g}$  and the kernel of the map is an abelian Lie algebra (\cite{BCG}, \cite{Mandal16}). In fact,  any Leibniz algebra $L$ can be viewed as an exact Courant algebra over the associated Lie algebra $L_{Lie}$ obtained by taking the quotient linear space ${L}/{L}^{ann}$, where $L^{ann}$ is also called the Leibniz kernel of $L$ is the module generated by elements of the form $[x,x]$ for $x \in L$. The notion of Leibniz algebras were introduced by A. Bloh \cite{Bloh65} under the name  $D$-algebras and it is rediscovered  by J.-L. Loday in connection with cyclic homology and Hochschild homology of matrix algebras (\cite{Lo1},\cite{Lo2}).  The (co) homology with coefficients in a representation associated to Leibniz algebras has been developed in \cite{LP} .  Any Lie algebra is also a Leibniz algebra, as in the presence of skew symmetry of the Lie bracket the Jacobi identity for the bracket is equivalent to the Leibniz identity.  It is also important to note that a Courant algebras can  be viewed as a model to get Leibniz algebras in the form of hemisemidirect products ( as introduced in \cite{KW} ) by starting with a Lie algebra $\mathfrak{g}$ and a $\mathfrak{g}$-modules. One of the Lie algebras  mostly encountered in geometry or in mathematical physics is the Lie algebra $\chi(M)$ of smooth vector fields on a smooth manifold $M$. This Lie algebra (which is also a Leibniz algebra ) can also be described as the space of linear derivations on the associative algebra $C^{\infty}(M)$ of smooth functions on the given manifold.  Here, we are interested in finding Leibniz algebras as a canonical extension of the Lie algebra $Der(A)$ of linear derivations on an arbitrary associative algebra $A$ both over the same coefficient ring $\mathbb{K}$. This leads us to consider a notion of Courant pair denoted simply by $(A,L)$. The deformation of algebraic and analytic structures are an important aspects if one studies their properties. They characterize the local behaviour in a small neighbourhood in the variety of a given type of objects.

A Courant pair $(A, L)$ consists of an associative algebra $A$ and a Leibniz algebra $L$ over the same coefficient field $\mathbb{K}$, equipped with a Leibniz algebra homomorphism  $\mu :L \longrightarrow \mathfrak{g}$, where $\mathfrak{g}$ denotes the Lie algebra of $\mathbb{K}$-linear derivations on $A$ under the commutator bracket. If we replace the Leibniz algebra in the definition by a Lie algebra then what we obtain is a Leibniz pair $(A, L)$ introduced by M. Flato et al. in \cite{LeibP}. The notion of Leibniz pair was defined as a generalization of the Poisson algebras and a formal deformation theory is studied in \cite{LeibP}. Later on it is found as a special case of a homotopy algebra introduced by H. Kajiura and J. Stashef in \cite{KS06}. They referred this as an open closed homotopy algebra (OCHA). A Lie algebroid or its algebraic counterpart, called a Lie- Rinehart algebra (\cite{Rin}) is also a special type of Leibniz pair by considering associative algebras also  to be a commutative algebra.  

A Courant pair appears naturally when one consider a Lie algebroid ( \cite{LieA} ) or more generally a Leibniz algebroid ( \cite{FreeN} ) ( or Loday algebroid, Courant algebroid ( \cite{CARoth}, \cite{FreeN}, \cite{BCG}, \cite{KW}, \cite{RW} ) ) over a smooth base manifold. In all these cases the underline Lie algebra of the Courant algebras are  considered to be the Lie algebra of vector fields or space of linear derivations of the associative algebra of functions on the manifold. More examples can be found if one allow Leibniz algebras in place of Lie algebra $L$ in the definition of a Leibniz pairs $ (A,L)$. Here, we consider this more general case of Leibniz pair in the name of Courant pair which will form a category containing the category of Leibniz pairs as a full subcategory. The cohomology space is different if one consider a Leibniz pair as an object in the larger category of Courant pairs. 

The paper is organized as follows:
In section $2$, we recall the Hochschild cohomology of an associative algebra with coefficient in a bimodule, a Leibniz algebra cohomology complex with coefficient in a representation and the notion of Courant algebras with examples.
 Section $3$ contains definition and various natural examples of the notion of Courant pairs. We discuss about representations or modules over a Courant pair as a suitable semi-direct products in the category of associative and Leibniz algebras respectively with corresponding modules.
In section $4$, we present the construction of a deformation bicomplex for a Courant pair $(A, L)$ from  Hochschild cohomology complex of the associative algebra $A$ and Leibniz cohomology complex for the Leibniz algebra  $L$.  In section $5$ we define one parameter formal deformation of a Courant pair, define the notion of equivalence of deformations, and study rigidity. In the last section we discuss about the obstruction cochains that arises in extending deformation of finite order with given infinitesimal. We mentioned examples to show that one can find many more deformations of a Leibniz pair when it is deformed as a Courant pair by using the  new bicomplex cohomology.

One of the main results in any deformation theory is to prove that obstruction cochains are cocycles. In our context, this consists of three parts; one arising from deformations of the considered associative algebra, one from the deformation of the Leibniz algebra  and the other from deformation of the Leibniz homomorphism  of a Courant pair. In \cite{G1}, M. Gerstenhaber showed that the obstruction cochains arising from deformations of associative algebras are cocycles and in \cite{B }  D. Balavoine showed  the obstruction cochains arising from deformations of Leibniz algebras are cocycles. The other part is done by a direct computation and given in an appendix.

Throughout we will consider vector spaces over a field $\mathbb{K}$ of characteristics zero and all maps are $\mathbb{K}$- linear unless otherwise it is specified . 

\section{Preliminaries}

%In this section, we will consider $\Kmath$ as a field of characteristic zero.
\subsection{Associative algebras and its cohomology}

\begin{Def}
 Let $A$ be an associative algebra over $\mathbb{K}$, a bimodule $M$ over $A$ or, an $A$-bimodule is a  $\mathbb{K}$-module equipped with two actions (left and right) of $A$
$$ A\times M \longrightarrow M ~~\mbox{and}~ M\times A \longrightarrow M$$ such that 
$(a m) a^\prime=a (m a^\prime)$ for $a,a^\prime \in A$ and $m \in M$.

The actions of $A$ and $\mathbb{K}$ on $M$ are compatible, that is $(\lambda a) m=\lambda (a m)=a (\lambda  m)$ for $\lambda \in \mathbb{K},~a\in A ~and ~m \in M$. When $A$ has the identity $1$ we always assume that $1 m=m 1=m ~~\mbox{for}~ m \in M$.
\end{Def}
Given a bimodule $M$ over $A$, the  Hochschild cochain complex of $A$ with coefficients in $M$ is defined as follows.
Set $C^n(A,M)= Hom_\Kmath (A^{\otimes n},M)~; n\geq 0$ where $A^{\otimes n}=A \otimes \cdots \otimes A~(n\mbox{-copies})$.
Let $\delta_H:C^n(A,M) \rightarrow C^{n+1}(A,M)$ be the $\mathbb{K}$- linear map given by 
\begin{equation*}
 \begin{split}
\delta_H f(a_1,a_2,\ldots,a_{n+1})&=a_1.f(a_2,\ldots,a_n)+\sum_{i=1}^{n}(-1)^i (a_1,\ldots,a_i a_{i+1},\ldots,a_{n+1})\\
&~~~~~~~~~~~~~~~~~~~~~~+(-1)^{n+1} (a_1,\ldots,a_{n-1},a_{n}).a_{n+1} 
\end{split}
\end{equation*} 
Then $\delta_H^2=0$ and the complex $(C^{*}(A,M),\delta_H)$ is called the Hochschild complex of $A$ with coefficients in the $A$-bimodule $M$.
 When $M=A$, where the actions are given by algebra operation in $A$ we denote the complex $C^{*}(A,A)$ by $C^{*}(A)$. The graded space $C^*(A)$ is a differential graded Lie algebra (or DGLA in short ) (see \cite{G1}).
\begin{Def}
A graded Lie algebra $\mathfrak{L}$ is a graded module $\mathfrak{L}=\{L_i\}_{i\in \mathbb{Z}}$ together with a linear map of degree zero, $[-,-]:\mathfrak{L}\otimes \mathfrak{L}\longrightarrow \mathfrak{L}$, $x\otimes y\mapsto [x,y]$ satisfying
\begin{equation*}
 \begin{split}
  &(i)[x,y]=-(-1)^{|x||y|}[y,x] \mbox {(graded skewsymmetry)}\\
&(ii)(-1)^{|x||z|}[x,[y,z]]+(-1)^{|y||x|}[y,[z,x]]+(-1)^{|z||y|}[z,[x,y]]=0 ~~\mbox{( graded Jacobi identity )}\\
 \end{split}
\end{equation*}
for $x,y,z \in \mathfrak{L}$, where $|x|$ denotes the degree of $x$.

A differential graded Lie algebra is a graded Lie algebra equipped with a differential $d$ satisfying 
$$d[x,y]=[dx,y]+(-1)^{|x|}[x,dy].$$
\end{Def}
\begin{Rem}\label{DGLA}
The shifted Hochschild complex $C^{*}(A,A)=C^{*+1}(A, A)$ is a DGLA where the graded Lie bracket is also called the Gerstenhaber bracket. 

Let $\phi$ be a Hochschild $p$-cochain and $\theta$ be a Hochschild $q$-cochain. The Gerstenhaber bracket of $\phi$ and $\theta$ is the $(p+q-1)$-cochain defined by
$$[\phi,\theta]=\phi\circ\theta-(-1)^{(p-1)(q-1)}\theta\circ \phi,$$
where
$$\phi\circ\theta(a_1,\cdots,a_{p+q-1})=\sum_{i=0}^{p-1}(-1)^{i(q+1)}\phi(a_1,\cdots,a_i,\theta(a_{i+1},\cdots,a_{i+q}),a_{i+q+1},\cdots,a_{p+q-1}).$$

Also, note that if $\alpha_0$ is the associative multiplication, then $\delta_H$ can be written in terms of the Gerstenhaber bracket and the multiplication $\alpha_0$:
$$\delta_H \phi=(-1)^{|\phi|}[\alpha_0,\phi].$$
\end{Rem}
\subsection{Leibniz algebra and its cohomology}
Now, we recall the definition of a Leibniz algebra and describe its cohomology.

\begin{Def} A  Leibniz algebra (left Leibniz algebra) is a $\Kmath$-module $L$, equipped with a bracket
 operation, which is $\mathbb{K}$-bilinear and satisfies the Leibniz identity,
 $$
 [x,[y,z]]=[[x,y],z] + [y, [x,z]] \quad for \quad x,y,z \in L.
 $$
\end{Def}
In the presence of antisymmetry of the bracket operation, the Leibniz identity is
equivalent to the Jacobi identity, hence any Lie algebra is a Leibniz
algebra.
\begin{Rem}
For a given Leibniz algebra $L$ as defined above,  the left adjoint operation $[x,-]$ is a derivation on $L$ for any
$x \in L.$  Analogously, Leibniz algebra ( right Leibniz algebra ) can be defined by requiring that the right adjoint map $[-,x]$ is  a derivation for any $x \in L.$ In that case, the Leibniz identity appeared in the above definition would be of the form: $$[x,[y,z]]= [[x,y],z]-[[x,z],y] ~~\mbox{for}~x,~y,~z \in L.$$
 Sometime, in the literature ( e.g., \cite{Lo1, Lo2, LP}) right Leibniz algebra has been considered.  All our Leibniz algebras in the present discussion will be left Leibniz algebras unless otherwise stated.
\end{Rem}
\begin{Exm}\label{Lex1}
Let $(\mathfrak{g},[-,-])$ be a Lie algebra and  $V$ be a $\mathfrak{g}$-module with the action $~(x,v)\mapsto x.v$. Take $L=\mathfrak{g}\oplus V $ with the bracket given by
$$[(x,v),(y,w)]_L=([x,y],x.w)~~\mbox{for}~~(x,v), (y,w) \in L .$$
Then $(L,[-,-]_L)$ is a Leibniz algebra. This bracket is called {\it hemisemidirect} product in \cite{KW}.
\end{Exm}

This shows that one can associate a Leibniz algebra $L$ to a smooth manifold $M$. If $\mathfrak{g}$ is the Lie algebra of Vector fields $\chi(M)$ over a smooth manifold $M$, then $C^{\infty}(M)$ is a $\chi(M)$-module with the left action $\chi(M) \times C^{\infty}(M) \rightarrow C^{\infty}(M):~(X,f)\mapsto X(f)$. It follows that  $L= \chi(M)\oplus C^{\infty}(M)$ is a Leibniz algebra with the bracket given by
$$[(X,f),(Y,g)]=([X,Y],X(f)).$$
 There are various other sources to generate more examples of  Leibniz algebras which may not be a Lie algebra.
\begin{Exm}
Let $A$ be an associative $\Kmath$-algebra equipped with a $\Kmath$-linear map $D: A\rightarrow A$, such that $D^2=D.$ 
Define a bilinear map $[-,-]: A\otimes A \rightarrow A$ by 
$$[x,y]:=(Dx)y-y(Dx)~~~ \mbox{for all}~ a,b\in A. $$
Then $(A,[-,-])$ is a left Leibniz algebra. In general, $A$ with the above bracket is not a Lie algebra unless $D=id.$
\end{Exm}

\begin{Exm}
Let $(L,d)$ be a differential Lie algebra with the Lie bracket $[-,-]$. Then $L$ is a Leibniz algebra with the bracket operation $[x,y]_d:= [dx,y]$. The new bracket on $L$ is  called the derived bracket(see \cite{DerB}).
\end{Exm}
One can find more examples appeared in  \cite{Lo1, Lo2, LP} , \cite {HM02, JMF09}.

\begin{Def}
Suppose  $L$ and $L^{\prime}$ are Leibniz algebras, a linear map $\phi: L \rightarrow L^{\prime}$ is called a homomorphism of Leibniz algebras if it preserves the Leibniz bracket, i.e.
$$\phi([x,y]_L)= [\phi(x), \phi (y)]_{L^{\prime}} ~~\mbox{for all}~x,y \in L. $$
\end{Def}

Leibniz algebras with Leibniz algebra homomorphisms form a category of Leibniz algebras, which contains the category of Lie algebras as a full subcategory.
%\begin{Def}

Let $L$ be a Leibniz algebra and $M$ be a representation of $L$. By
definition (\cite{LP,JMF09}), $M$ is a $\Kmath$-module equipped with two actions (left and
right) of $L$,
$$
[-,-]:L \times M \rightarrow M \qquad \text{and} ~~~[-,-]:M \times L
\rightarrow M
$$
such that for $m \in M$ and $x,y \in L$ following hold true,
\begin{equation*}
\begin{split}
[m,[x,y]] =~ & [[m,x],y] + [x,[m,y]]\\
[x,[m,y]]=  ~& [[x,m],y] + [m,[x,y]]\\
[x,[y,m]]= ~ &[[x,y],m] + [y,[x,m]].
\end{split}
\end{equation*}
%\end{Def}
Define $CL^n(L;M):=\text{Hom}_{\Kmath}(L^{\otimes n},M)$, $n \geq 0$.
Let
$$
\delta^n: CL^n(L;M) \rightarrow CL^{n+1}(L;M)
$$
be a $\Kmath$-homomorphism defined by
\begin{multline*}\label{Leibniz coboundary}
(\delta^n \psi)(X_1,X_2, \cdots , X_{n+1}) \\=\sum_{i=1}^{n} (-1)^{i-1} [X_i,\psi(X_1, \cdots, \hat{X_i}, \cdots, X_{n+1})]
 +(-1)^{n+1} [\psi(X_1, \cdots, X_{n}), X_{n+1}]  \\+
 \sum_{1\leqslant i < j \leqslant n+1} (-1)^{i}
 \psi(X_1, \cdots, \hat{X_i},\cdots, X_{j-1},[X_i,X_j],X_{j+1}\cdots, X_{n+1}).
\end{multline*}
Then $(CL^*(L;M),\delta)$ is a cochain complex, whose cohomology is denoted by $HL^*(L; M)$, is
called the cohomology of the Leibniz algebra $L$ with coefficients in
the representation $M$ (see \cite{LP, JMF09}). In particular, $L$ is a representation of
itself with the obvious actions given by the Leibniz algebra bracket of $L$.

\subsection{Courant algebras over a Lie algebra}
Now, we recall the definition of  Courant algebras  from \cite{BCG}. Let $\mathfrak{g}$  be a Lie algebra over $\mathbb{K}$.
\begin{Def}\label{Courant algebra}
A {\it Courant algebra }  over a Lie algebra $\mathfrak{g}$  is a Leibniz algebra $L$ equipped with a homomorphism of Leibniz algebras $\pi: L\rightarrow\mathfrak{g}$. We will simply denote it as $\pi: L \longrightarrow \mathfrak{g}$ is a Courant algebra.
\end{Def}

\begin{Exm}
A Courant algebroid over a smooth manifold $M$  gives an example of a Courant algebra over the Lie algebra of vector fields $\mathfrak{g}=\Gamma(TM)$ by taking $\mathfrak{a}$ as the Leibniz algebra structure on the space of sections of the underlying vector bundle of the Courant algebroid.  From \cite{RW} it follows that any Courant algebra is actually an example of a $2$- term $L_\infty$ algebra \cite{BC,St}.
\end{Exm}

\begin{Def}\label{exact Courant algebra}
An  {\it exact} Courant algebra  over a Lie algebra $\mathfrak{g}$ is a Courant algebra $\pi: L \longrightarrow \mathfrak{g}$ for which $\pi$ is a surjective linear map and $\mathfrak{h}=ker (\pi)$ is abelian, i.e. $[h_1,h_2]=0$ for all $h_1,h_2 \in \mathfrak{h}$.
\end{Def}
For an exact Courant algebra $\pi: L \longrightarrow \mathfrak{g}$, there are two canonical actions of  $\mathfrak{g}$ on $\mathfrak{h}$: 
$[g,h]=[a,h]$ and $[h,g]=[h,a]$ for any $a$ such that $\pi(a)=g$. Thus the $\mathbb{K}$-module $\mathfrak{h}$ equipped with these two actions denoted by the same bracket notation $[-,-]$, is a representation of $\mathfrak{g}$ ( viewed as a Leibniz algebra ). 

The next example will give a natural exact Courant algebra associated with any representation of the Lie algebra $\mathfrak{g}$.

\begin{Exm} As in Example \ref{Lex1}, $\mathfrak{g}$ be a Lie algebra acting on the vector space $V$. Then
$L =\mathfrak{g}\oplus V$ becomes a Courant algebra over $\mathfrak{g}$ with the Leibniz algebra homomorphism given by projection on $\mathfrak{g}$  and the Leibniz bracket on $L$ given by
$$ [(g_1,h_1),(g_2,h_2)]=([g_1,g_2], g_1. h_2),$$
where $g.h$ denotes the action of the Lie algebra $\mathfrak{g}$ on $V$.
\end{Exm}
In \cite {Mandal16} it is shown that exact Courant algebras over a Lie algebra $\mathfrak{g}$ can be characterised via Leibniz $2$- cocycles, and the automorphism group of a given exact Courant algebra is in a one-to-one correspondence with first Leibniz cohomology space of $\mathfrak{g}$. Exact Courant algebras also appeared in the general study of Leibniz algebra extension and a discussion of some unified product for Leibniz algebras is in \cite{AgMi13, Mil15}.

\begin{Exm}

If we consider  a Leibniz representation $\mathfrak{h}$ of the Lie algebra $\mathfrak{g}$, then
$L =\mathfrak{g}\oplus \mathfrak{h}$ becomes a Courant algebra over $\mathfrak{g}$ via the bracket
$$ [(g_1,h_1),(g_2,h_2)]=([g_1,g_2], [g_1, h_2]+[h_1,g_2]),$$
where the actions (left and right) of $\mathfrak{g}$ on $\mathfrak{h}$ are denoted by the same bracket $[-,-]$.
\end{Exm}

\section{Courant pairs}
In this section we introduce the notion of a Courant pair and discuss about various natural examples . We define  representation or modules of a Courant pair as a suitable semi-direct product in the category of associative and Leibniz algebras respectively with corresponding modules.

Let $A$ be an associative algebra over $\mathbb{K}$ and by $\mathfrak{g}$, we denote the Lie algebra $(Der_{\mathbb{K}}(A), [-,-]_c)$ of linear derivations of $A$ with commutator bracket  $[-,-]_c$.
 Now we define a Courant pair as a Courant algebra over the Lie algebra $\mathfrak{g}$ which is canonically associated to a given associative algebra $A$.
\begin{Def}
 A Courant pair  $ (A,L)$ consists of an associative algebra $A$ and a Leibniz algebra $L$  over the same  coefficient ring $\mathbb{K}$, and equipped with  a Leibniz algebra homomorphism
  $\mu: L\longrightarrow \mathfrak{g} $.
\end{Def}

\begin{Rem}\label{RemC}
It follows from the definition that a Courant pair $(A, L )$ can be expressed as a triplet of the involved algebra operations and homomorphism $(\alpha,\mu,\lambda)$. Here,  $\alpha : A\times A\rightarrow A$, denotes the associative multiplication map on $A$, $\lambda : L\times L\rightarrow L $ denotes the Leibniz algebra bracket in $L$ and the homomorphism $\mu$ defines an action $\mu: L\times A\rightarrow A$ given by $\mu(x,a)=\mu(x)(a)$ for all $x \in L$ and $a \in A$.
\end{Rem}
Every Leibniz pair  defined in \cite{LeibP} is an example of a Courant pair. In fact, if $(A, L)$ is a Leibniz pair then by definition, $A$ is an associative algebra and $L$ is a Lie algebra over some common coefficient ring $\Kmath$, connected by a Lie algebra morphism $\mu: L \rightarrow \mathfrak{g}$. And since the category of Lie algebras is a subcategory of the Leibniz algebra category, it follows that a Leibniz pair is also a Courant pair. In particular,
\begin{enumerate}
  \item a Poisson algebra $A$ gives a Courant pair $(A,A)$ with the Leibniz algebra homomorphism $\mu=id$;
  \item any Lie algebroid $(E,[-,-], \rho)$ gives a Courant pair $(C^{\infty}(M),\Gamma E)$ with  the Leibniz algebra homomorphism $\mu = \rho$;. 
  \item any Lie- Rinehart algebra $(A, L)$ is also a Courant pair satisfying the additional conditions that A as a commutative algebra, $L$ as an $A$-module and $\mu: L \rightarrow \mathfrak{g}$ as an $A$-module morphism such that $[x,fy]=f[x,y]+\mu(x)(f)y;$ for $x,y\in L$ and $f\in A$. 
   \end{enumerate}
\begin{Exm}
Let A be an associative commutative algebra and a module over the Lie algebra $\mathfrak{g}=(Der(A),[-,-]_c)$, where  $\mathfrak{g}$ acts on $A$ via derivations. As in  Example (2.4), if we consider the Leibniz algebra $L = Der(A)\oplus A$ , then $(A,Der(A)\oplus A)$ is a courant pair, where the Leibniz algebra morphism $\mu$ is given by
the projection map from $Der(A)\oplus A$ onto $Der(A)$.
\end{Exm}
\begin{Rem}
In particular, for $A$ to be the space of smooth functions on a smooth manifold this type of examples of Courant pair appears as twisted Dirac structure of order zero in \cite{ACR}.
\end{Rem}
In the sequel, one can find  Courant pairs associated to an algebra by considering modules and  flat connections on the module of the algebra.  

Let $A$ be an associative algebra and $V$ be an unitary $A$-module. A derivation law, or Koszul connection on $V$ is an $A$-linear mapping $\nabla: Der(A) \rightarrow Hom_R(V,V)$ ( the image of $X\in Der(A)$ is denoted by $\nabla_X$) such that
$\nabla_X(f.v)=f.\nabla_X(v)+X(f).v$
for $X\in Der(A),~ f\in A,~\mbox{and}~ v\in V$. Further, the connection $\nabla$ is called flat connection if 
$\nabla_{[X,Y]}= [\nabla_X,\nabla_Y]_c.$
Now, $(Der(A),[-,-]_c)$ is a Lie algebra. If $\nabla$ is a flat connection on $V$, then $V$ is a $Der(A)$-module. Therefore, hemisemidirect product of $Der(A)$ and $V$ defines a Leibniz bracket on the A-module $Der(A)\oplus V$ as
  $$[(X,v),(Y,w)]=([X,Y],\nabla_X(w)),$$
where $v,w\in V$. Therefore, we have a Courant pair $(A,Der(A)\oplus V)$ with Leibniz algebra morphism $\mu$ is given by the projection map onto $Der (A) $.   

Now, for a Lie algebroid $(E,[-,-],\rho)$ over $M$ and a flat $E$-connection (\cite{LieA}) on a Vector bundle $F\rightarrow M$, there is an associated Courant pair $(A, L)$ where $A$ is algebra of smooth functions on $M$ and $L$ is the hemisemidirect product Leibniz algebra obtained from  the space of sections of the Lie algebroid and the vector bundle respectively.

\begin{Exm}
 Let $(E,[-,-],\rho)$ be a Lie algebroid over a manifold $M$. An $E$-connection on a vector bundle $F\rightarrow M$ is a $\Kmath$-linear map 
$$\nabla: \Gamma E \rightarrow Hom_{\Kmath}(\Gamma F,\Gamma F)$$ 
satisfying following properties
$$\nabla_{(f.X)}(\alpha)= f.\nabla_{X}(\alpha),$$
$$\nabla_X(f.\alpha)= f.\nabla_X(\alpha)+ \rho(X)(f).\alpha,$$
for $X\in \Gamma E$, $f\in C^{\infty}(M)$, and $\alpha \in \Gamma F$. Here, $\nabla(X)$ is denoted by $\nabla_X$. An $E$-connection is called flat if $\nabla_{[X,Y]}=[\nabla_X,\nabla_Y]_c$. A flat $E$-connection on a vector bundle $F\rightarrow M$ is also called a representation of E on $F\rightarrow M$ (recall  from \cite{LieA}). Now, $\Gamma E$ is a Lie algebra and $\Gamma F$ is a $\Gamma E$-module (it follows from flatness of $\nabla$). Therefore, by taking hemisemidirect product of $\Gamma E$ and $\Gamma F$, we have a Leibniz algebra bracket on
$\Gamma E \oplus \Gamma F$ given by
$$[(X,\alpha),(Y,\beta)]=([X,Y],\nabla_X(\beta)) $$ 
for $X,Y \in \Gamma E$ and $\alpha,~ \beta\in \Gamma F $. Consider $A =C^{\infty}(M)$, and $L=\Gamma E \oplus \Gamma F$, then $(A,L)$ is a Courant pair with the Leibniz algebra morphism:
$$\mu :\Gamma E \oplus \Gamma F \rightarrow \chi(M)=Der(A)$$ 
where $\mu=\rho \circ \pi_{1}$, and $\pi_{1}$ is the projection map onto $\Gamma E$. 
\end{Exm} 

\begin{Exm}
Let $R$ be a commutative unital ring and $A$ be a commutative unital $R$-algebra.  Recall from  \cite{FreeN},  a Leibniz pseudoalgebra $(\mathcal{E},[-,-],\rho)$ over $(R,A)$ consists of an $A$-module $\mathcal{E}$, an $A$-module homomorphism $\rho:\mathcal{E}\rightarrow Der_R(A)$ and a Leibniz $R$-algebra structure $[-,-],$ on $\mathcal{E}$  such that
$$[X,fY]=f[X,Y]+\rho(X)(f)Y,$$
for all $f\in A$ and $X,Y\in \mathcal{E}$.  
Every Leibniz pseudoalgebra $(\mathcal{E},a,[-,-])$ over $(R,A)$ gives a Courant pair $(A,\mathcal{E})$. 

In particular, any Leibniz algebroid $E$ over a smooth manifold $M$ is a Leibniz pseudoalgebra over $(\mathbb{R},C^{\infty}(M))$. So, it gives a Courant pair $( C^{\infty}(M), \Gamma(E) ).$
\end{Exm}
\begin{Exm}
Suppose $\mathcal{A}$ is a associative and commutative algebra. Let $\Omega^1_{\mathcal{A}}$ be the $\mathcal{A}$-module of K$\ddot{\mbox{a}}$hler differentials(\cite{Khlr}), with the universal derivation $$ d_{0} : \mathcal{A}\rightarrow \Omega^1_{\mathcal{A}}.$$ Consider the $\mathcal{A}$-module $Der(A)\oplus \Omega^1_{\mathcal{A}}$ equipped with
the Leibniz bracket  given by $$[(X,\alpha),(Y,\beta)]=([X,Y],L_X\beta-i_Y d_0 \alpha).$$
 Then $(A,Der(A)\oplus \Omega^1_{\mathcal{A}})$ is a Courant pair. This shows that a Courant pair is also appeared in the context of Courant -Dorfman algebras introduced in \cite{CHRtb}.
\end{Exm}

\begin{Rem}
A Courant-Dorfman algebra (\cite{CHRtb}) has a underlined Leibniz pseudoalgebra (or Leibniz-Rinehart algebra )  and hence a Courant pair structure in it.
\end{Rem}

\begin{Exm}
Any Courant algebroid $(C,[.,.],\rho,<.,.>)$ (algebraic counterpart defined in \cite{CARoth}) is a Leibniz pseudoalgebra  $(C,[.,.],\rho)$. So, any Courant algebroid has a underlined Courant pair.
\end{Exm}

\subsection{Modules of Courant pairs:}
 Suppose $(A, L)$ is a Courant pair. We define modules of $(A, L)$ as a suitable Courant algebra by considering semi- direct products of associative and Leibniz algebras with respective modules of these algebras.
 Suppose that $M$ is an $(A, A)$-bimodule and $P$ an $L$-module. We denote by $ A+M$ and $L+P$, respectively, the associative and Leibniz semi-direct products. Now we define modules over a Courant pair, using the following general principle: if $(A,L)$ is a Courant algebra, then an $(A,L)$ -module is a pair $(M,P)$ such that $(A\oplus M, L\oplus P)$ is also a Courant pair and contains (A,L) as a subobject and $(M,P)$ as an abelian ideal in an
appropriate sense.

\begin{Def}
 A module over a Courant pair $(A,L)$  means a pair $(M,P)$, where $P$ is a Leibniz algebra module over $L$, $M$ is an $(A, A)$-bimodule, and there is a Leibniz algebra homomorphism $\tilde{\mu}:~(L+P) \longrightarrow Der_{\mathbb{K}}(A+M)$ extending the Leibniz algebra morphism $\mu:~L\rightarrow DerA $ in the Courant pair $(A, L)$ in the following sense:
 \begin{enumerate}
 \item $\tilde{\mu}(x,0)(a,0)=\mu(x)(a)$   for  any $x\in L,~a\in A.$\\
 \item $\tilde{\mu}(x,0)(0,m),~\tilde{\mu}(0,\alpha)(a,0)\in M$   for any $x\in L,~ a\in A,~\alpha\in P,~m\in M. $\\
 \item $\tilde{\mu}(0,\alpha)(0,m)=0$ for any $\alpha\in P,~ m\in M.$
 \end{enumerate}
 \end{Def}
 
 \begin{Rem}
 Every Courant pair $(A, L)$ is a module over itself. This fact will be useful in the later section while considering deformations. We may also canonically define morphisms of Courant pairs and hence morphisms of modules of a given Courant pair $(A, L)$.
 \end{Rem}

 \begin{Note}\label{phi}

 It is important to notice here that for each $\alpha \in P$  the map $A \longrightarrow M$ defined by sending $a$ to $\tilde{\mu}(0,\alpha)(a,0)$ is a derivation of $A$ into $M$. This yields  a  map $\phi: P\rightarrow Der(A,M)$, which we will use in the double complex defined later.
\end{Note}
 
\section{Deformation complex of a Courant pair:}
In the present section, we introduce the deformation complex of a Courant pair. We shall observe subsequently that the second and the third cohomologies associated to this cohomology complex encode all the  information about deformations.
 Suppose (M,P) is a module of the Courant pair $ (A, L )$. So $M$ is an $A$-bimodule and $P$ is a Leibniz-$L$-module. We need to consider the Hochschild complex $C^{*}(A, M)$ and Leibniz algebra complex $C^{*}( L,P )$ in the sequel.  Here we will denote the tensor modules $\otimes^{p}A$ simply by $A^p$.

 Let $C^p(A,M)=Hom(A^p,M)$, the $p$-th Hochschild cochain group of $A$ with coefficients in $M$. Note that $C^0(A,M)=M$. Now $C^p=C^p(A,M)$ is  an $L$ - module (symmetric Leibniz algebra module) by the actions given as
 $$[x,f](a_1,a_2,...,a_p)=[x,f(a_1,a_2,...,a_p)]-\sum_{i=1}^p f(a_1,...,[x,a_i],...,a_p)$$
 and $[f,x]=-[x,f]$ for $f\in C^p$, $a_1, \cdots,a_p \in A$ and $x\in L $.

 Here, $[x,m]=-[m,x]=\tilde{\mu}(x)(m)$
%$~[p,a]=-[p,a]=\tilde{\mu}(p)(a)~$
 for$~a\in A,~m\in M,$ and $~x\in L $.
 One can check  that,
 $$[[x,y],f]=[x,[y,f]]-[y,[x,f]],$$
for $x,y \in L$ and $f \in C^p$.
 It follows that the Lie algebra $Der(A,M)$ of linear derivation from $A$ to $M$ is a Leibniz algebra submodule of $C^1(A,M)$ and the mapping $\phi: P\rightarrow Der(A,M)$ defined in the previous section in ( \ref{phi}) is a Leibniz algebra module homomorphism.
In otherwords,
 $[x,f]\in Der(A,M)$ for $x\in L$ and $f\in Der(A,M).$
 Also, $$\phi[x,p]=[x,\phi(p)]~\mbox{(using the Leibniz algebra structure on the semi-direct product space $L+P$)}.$$
 Recall the Hochschild coboundary $\delta_H:~C^p\rightarrow C^{p+1}$ is given by
\begin{equation*}
\begin{split}
&(\delta_Hf)(a_1,...,a_{p+1})\\                                             
=&a_1f(a_2,...,a_{p+1})+\sum_{i=1}^p(-1)^if(a_1,...,a_i.a_{i+1},...,a_{p+1})+(-1)^{p+1}f(a_1,...,a_p)a_{p+1}.        
\end{split}
\end{equation*}
Now it follows that $\delta_H$ is a Leibniz algebra module homomorphism by using the definitions of the map $\delta_H$, and $L$-module action on $C^p$.
% i.e. $$\delta_H([x,f])=[x,\delta_Hf].$$
\begin{Prop}
 Let (M,P) be a module over a Courant pair (A,L). Then $\delta_H:~C^p\rightarrow C^{p+1}$ is a homomorphism of Leibniz algebra modules over L, i.e. $\delta_H([x,f])=[x,\delta_Hf]$ for $x\in L$ and $f\in C^p$.
\end{Prop}

Let $(M,P)$ be a module of the Courant pair $(A,L)$. We now define  a double complex for $(A,L)$ with coefficients in the module $(M,P)$  as follows:

Set $C^{p,q}( A,L )=Hom_{\mathbb{K}}(L^q,C^p(A,M))\cong Hom_{\mathbb{K}}(A^p\otimes L^q,M)$ for all  $(p,q)\in \mathbb{Z}\times \mathbb{Z}$ with $p>0$, q$\geq$0,
  where $C_{0,q}( A,L )=Hom_{\mathbb{K}}(L^q,P)$. For $p>0$, we consider the vertical coboundary map $C^{p,q}( A,L )\rightarrow C^{p+1,q}( A,L )$ is the Hochschild coboundary map $\delta_H$. For $p=0$ and for all $q$, Leibniz algebra module $C^1(A,M)$ the map $\delta_v:C^{0,q}( A,L )\rightarrow C^{1,q}( A,L )$. For $p>0$, vertical maps are Hochschild coboundaries, and $Der(A,M)$ the kernel of $\delta_H:C^1(A,M)\rightarrow C^2(A,M).$ Therefore, the composition of two vertical maps is zero for all $p$ and $q$.
\[
\begin{CD}
....@. .... @. ....@. ....\\
@A\delta_H AA @A\delta_H AA  @A\delta_H AA \\
Hom(A^2,M) @>\delta_L>> Hom(L,C^2(A,M)) @>\delta_L>> Hom(\otimes^2L,C^2(A,M))
@>\delta_L>>  ....
\\
@A\delta_H AA @A\delta_H AA  @A\delta_H AA \\
Hom(A,M) @>\delta_L>> Hom(L,C(A,M)) @>\delta_L>> Hom(\otimes^2L,C^2(A,M))
@>\delta_L >> ....\\
@A\delta_v AA @A\delta_v AA @A\delta_v AA\\
Hom({\mathbb{K}},P) @>\delta_L>> Hom(L,P) @>\delta_L>> Hom(\otimes^2L,P)
@>\delta_L>> ....\\
\end{CD}
\]
In the horizontal direction, we have for all $p$ and $q$ the Leibniz algebra coboundary $\delta_L: C^{p,q}( A,L )\longrightarrow C^{p,q+1}( A,L ).$ Since all the vertical coboundaries are Lie module homomorphisms, it follows that they commute with all the horizontal ones. Therefore, we can  define a total cochain complex as follows: For $n\geq 0$.
 $$C^n_{tot}( A,L )=\bigoplus _{p+q=n}C^{p,q}( A,L )~\mbox{and}~\delta_{tot}:C^n_{tot}( A,L )\rightarrow C^{n+1}_{tot}( A,L ),$$ whose restriction to $C^{p,q}( A,L )$ by definition is $\delta_H+(-1)^p\delta_L$. The resulting cohomology of the total complex we define to be the cohomology of the given Courant pair $(A,L)$ with coefficients in the module $(M,P)$, and we denoted this cohomology as $HL^* (A,L; M,P)$.

In particular for $(M,P)= (A,L)$, with the obvious actions the cohomology is denoted by $HL^* (A,L; A,L)$. We will recall this cohomology in the next section for the deformation of Courant pairs. 

\begin{Rem}
The cohomology spaces are different if one consider a Leibniz pair as an object in the larger category of Courant pairs. Recall from \cite{LeibP},
 $H^*_{LP}(A,L; M,P)$ is the cohomology space for Leibniz pair $(A,L)$ with coefficients in the module $(M,P)$. 
 Consider the Leibniz pair $(\mathbb{R},\chi(\mathbb{R}^n))$  with structure map $\mu$ as zero map and $(\mathbb{R},\mathbb{R})$ as a module over this Leibniz pair.
It follows that $H^*_{LP}(\mathbb{R},\chi(\mathbb{R}^n);\mathbb{R},\mathbb{R}) \cong H^*(\chi(\mathbb{R}^n),\mathbb{R})$.

On the other hand  if we take $(\mathbb{R},\chi(\mathbb{R}^n))$ as a Courant pair and $(\mathbb{R},\mathbb{R})$ as a module over it, then  $HL^*(\mathbb{R},\chi(\mathbb{R}^n);\mathbb{R},\mathbb{R})\cong HL^*(\chi(\mathbb{R}^n),\mathbb{R})$.

We know that the Leibniz algebra cohomology space $HL^*(\chi(\mathbb{R}^n),\mathbb{R})$ and the Lie algebra cohomology space $H^*(\chi(\mathbb{R}^n),\mathbb{R})$  are different as there are new generators for the dual Leibniz algebra structure over the Leibniz cohomology space (see \cite{LeibV}, \cite{Lod}).

This shows that, while considering a Leibniz pair $(\mathbb{R},\chi(\mathbb{R}^n))$ as an object in the category of Courant pairs, we have a somewhat different cohomology space.
\end{Rem}

\section{Deformation of Courant pairs:}

In this section, we study formal deformation of Courant pairs. All the basic notions of deformation theory of algebraic structures are originally due to M. Gerstenhaber(\cite{G1}, \cite{G2}, \cite{G3}, \cite{G4} ).  Here we briefly describe the analogous concepts related to deformation of a Courant pair. This can be viewed as a generalization of the deformation of Leibniz pairs (in \cite{LeibP}) to a larger category by allowing more deformations.

Let $(A,L)$ be a Courant pair, then using remark (3.2), it can be viewed as a triplet $(\alpha,\mu,\lambda)$. Recall that  a formal $1$-parameter family of deformations of an associative algebra $A$ is an associative algebra bracket $\alpha_t$ on the $\mathbb{K}[[t]]$-module $A_t=A{\otimes}_{\mathbb K} \mathbb{K}[[t]]$, where $\alpha_t=\sum_{i\geq 0}\alpha_i t^i,~\alpha_i \in C^2(A;A)$ with $\alpha_0$ being the algebra bracket on $A$. Analogously,  a formal $1$-parameter family of deformations of  a Leibniz algebra $L$ is a Leibniz bracket $\lambda_t$ on the $\mathbb{K}[[t]]$-module $L_t=L{\otimes}_{\mathbb K} \mathbb{K}[[t]]$, where $\lambda_t=\sum_{i\geq 0}\lambda_i t^i,~\lambda_i \in CL^2(L;L)$ with $\lambda_0$ denotes the original Leibniz bracket on $L$.

\begin{Def}
 A deformation  of a Courant pair $(A,L)$ whose structure maps are given by the triplet $(\alpha,\mu,\lambda)$ is defined as a tuple of the form $(\alpha_t,\mu_t,\lambda_t)$, where $\alpha_t$ is a deformation of $\alpha$, i.e., an associative $\mathbb{K}[[t]]$-bilinear multiplication $A[[t]]\times A[[t]]\rightarrow A[[t]]$ such that $\alpha_t=\sum_{i\geq0}\alpha_it^i$, where $\alpha_0=\alpha$, $\lambda_t$ is a deformation of Leibniz algebra bracket $\lambda$, and $\mu_t=\sum_{i\geq0}\mu_it^i$, where $\mu_0=\mu$. Also following compatibility conditions among $\alpha_t$, $\mu_t$ and $\lambda_t$ are satisfied:
\end{Def}

\begin{equation}\label{D1}
\mu_t(x,\alpha_t(a,b))=\alpha_t(a,\mu_t(x,b))+\alpha_t(\mu_t(x,a),b),
\end{equation}
\begin{equation}\label{D2}
\mu_t(\lambda_t(x,y),a)=\mu_t(x,\mu_t(y,a))-\mu_t(y,\mu_t(x,a)).
\end{equation}
By above definition we have also the following equations, using the fact that $\alpha_t$ and $\lambda_t$ are deformations of the associative multiplication $\alpha$ on $A$ and the Leibniz bracket $\lambda$ on $L$, respectively:
\begin{equation}\label{D3}
\alpha_t(a,\alpha_t(b,c))=\alpha_t(\alpha_t(a,b),c),
\end{equation}
\begin{equation}\label{D4}
\lambda_t(x,\lambda_t(y,z))=\lambda_t(\lambda_t(x,y),z)+\lambda_t(y,\lambda_t(x,z)).
\end{equation}
Now expanding both sides of the above equations and collecting coefficients of $t^n$, we get the following equations: For $a,b,c \in A$ and $x,y,z \in L$
\begin{align}
 \sum_{i+j=n}\mu_i(x,\alpha_j(a,b))&=\sum_{i+j=n}{\alpha_i(a,\mu_j(x,b))+\alpha_i(\mu_j(x,a),b)};\label{D5}\\
%~~~ ~\mbox{for}~ x\in L,~ a,b \in A;\\
   \sum_{i+j=n}\mu_i(\lambda_j(x,y),a)&=\sum_{i+j=n}{\mu_i(x,\mu_j(y,a))-\mu_i(y,\mu_j(x,a))};\label{D6}\\
%~~~~ \mbox{for}~ x,y\in L,~ a\in A;\\
   \sum_{i+j=n}\lambda_i(x,\lambda_j(y,z))&=\sum_{i+j=n}{\lambda_i(\lambda_j(x,y),z)+\lambda_i(y,\lambda_j(x,z))};\label{D7}\\
% ~~~~\mbox{for}~ x,y,z\in L;\\
   \sum_{i+j=n}\alpha_i(a,\alpha_j(b,c))&=\sum_{i+j=n}\alpha_i(\alpha_j(a,b),c)\label{D8}.
%~~~~~~ \mbox{for} ~a,b,c\in A.
\end{align}
\begin{Rem} For $n=0$, equations \eqref{D1} and \eqref{D2} imply that $\mu$ is a Leibniz algebra morphism from $L$ into $Der(A)$, and \eqref{D4}, \eqref{D3} implies the usual Leibniz identity for $\lambda_0$ and the associativity of the multiplication $\alpha_0$.
 Now,   the map $\delta_L: Hom(L^q,C^p(A,A))\rightarrow Hom(L^{q+1},C^p(A,A))$ is defined as:
\begin{align*}
  \delta_Lf(x_1,...,x_{q+1})=&(-1)^{q+1}\sum_{i=1}^q (-1)^{i+1}[x_i,f(x_1,..,\hat{x_i},..,x_{q+1})]+[f(x_1,...,x_q),x_{q+1}]\\
                             &+(-1)^{q+1}\sum_{1\leq i< j \leq q+1}f(x_1,..,\hat{x_i},..,x_{j-1},[x_i,x_j],x_{j+1},..,x_{q+1}).
\end{align*}
So, for $n=1$, equations \eqref{D8} and \eqref{D7} imply that $\delta_H(\alpha_1)=0,~~ \delta_L(\lambda_1)=0$ and equations \eqref{D5} and \eqref{D6} imply that $\delta_H\mu_1+\delta_L\alpha_1=0$ and
$\delta_v\lambda_1-\delta_L\mu_1=0$. Thus, for $n=1$, equations \eqref{D5}-\eqref{D8} give that 
 $\delta_{tot}(\alpha_1,\mu_1,\lambda_1)=0$. So $(\alpha_1,\mu_1,\lambda_1)$ is a $2$-cocycle in the double complex described in the previous section. 
\end{Rem}

\begin{Def} The $2$-cochain $(\alpha_1,\mu_1,\lambda_1)$ is called the infinitesimal of the deformation $(\alpha_t,\mu_t,\lambda_t)$. More generally, if $(\alpha_i,\mu_i,\lambda_i)=0$ for $1\leq i\leq(n-1)$, and $(\alpha_n,\mu_n,\lambda_n)$ is a non zero cochain in $C^2_{tot}$, then $(\alpha_n,\mu_n,\lambda_n)$ is called the $n$-infinitesimal of the deformation $(\alpha_t,\mu_t,\lambda_t)$.
\end{Def}

\begin{Prop} The infinitesimal $(\alpha_1,\mu_1,\lambda_1)$ of the deformation $(\alpha_t,\mu_t,\lambda_t)$ is a $2$-cocycle in $C^2_{tot}$. More generally, the $n$-infinitesimal is a $2$-cocycle.
\end{Prop}

% $d(\alpha_1,\mu_1,\lambda_1)$ = $(\delta_H\alpha_1,\delta_L\alpha_1+\delta_H\mu_1,-\delta_L+\delta_v\lambda_1,\delta_L\lambda_1) = 0.$
 
Let $(\alpha_t,\mu_t,\lambda_t)$ and $(\tilde{\alpha}_t,\tilde{\mu}_t,\tilde{\lambda}_t)$ be two deformations of Courant pair.

\begin{Def} We say that $(\alpha_t,\mu_t,\lambda_t)$ is equivalent to $(\tilde{\alpha}_t,\tilde{\mu}_t,\tilde{\lambda}_t)$, where $\tilde{\alpha}_t=\sum_{i\geq0}t^i\tilde{\alpha}_i$ and $\tilde{\alpha}_0=\alpha$, etc., if there exists $\mathbb{K}[[t]]$-linear maps $\Phi_t=id_A+\sum_{i\geq1}t^i \phi_i : A[[t]] \rightarrow A[[t]]$ (with each $\phi_i$ a linear map $A \rightarrow A$ extended to be $\mathbb{K}[[t]]$-linear map) and $\Psi_t=id_L+\sum_{i\geq1}t^i\psi_i:L[[t]] \rightarrow L[[t]]$ such that
\begin{align}
 \tilde{\alpha}_t(a,b)&=\Phi_t^{-1}\alpha_t(\Phi_ta,\Phi_tb),\label{D9}\\
 \tilde{\mu}_t(x,a)&=\Phi_t^{-1}\mu_t(\Psi_tx,\Phi_ta),\label{D10}\\
 \tilde{\lambda}_t(x,y)&=\Psi_t^{-1}\lambda_t(\Psi_tx,\Psi_ty).\label{D11}
\end{align}
\end{Def}

\begin{Def}
Any deformation of Courant pair equivalent to the deformation $(\alpha,\mu,\lambda)$ is said to be a trivial deformation.
\end{Def}

\begin{Thm} The cohomology class of the infinitesimal of the deformation $(\alpha_t,\mu_t,\lambda_t)$ of Courant pair is determined by the equivalence class of $(\alpha_t,\mu_t,\lambda_t)$.
\end{Thm}
\begin{proof}
Let $(\Phi_t,\Psi_t)$ : $(\alpha_t,\mu_t,\lambda_t) \rightarrow (\tilde{\alpha}_t,\tilde{\mu}_t,\tilde{\lambda}_t)$ be an equivalence of these two deformations. Then expanding and comparing coefficients of t in the equations \eqref{D9}, \eqref{D10}, \eqref{D11}, we get $\tilde{\alpha}_1-\alpha_1=\delta_H\phi_1$, $\tilde{\lambda}_1-\lambda_1=\delta_L\psi_1,$ and $\tilde{\mu}_1-\mu_1=\delta_v\psi_1+\delta_L\phi_1.$ So, it follows that $\delta_{tot}(\phi_1,\psi_1)=(\alpha_1,\mu_1,\lambda_1)-(\tilde{\alpha}_1,\tilde{\mu}_1,\tilde{\lambda}_1)$.
\end{proof}

\begin{Def}
A Courant pair is said to be rigid if every deformation of the Courant pair is trivial.
\end{Def}

\begin{Thm}
 A non trivial deformation of a Courant pair is equivalent to a deformation whose $n$-infinitesimal is not a coboundary for some $n\geq1$.
\end{Thm}

\begin{proof}
Let $(\alpha_t,\mu_t,\lambda_t)$ be a deformation of a Courant pair $(\alpha,\mu,\lambda)$ with $n$-infinitesimal $(\alpha_n,\mu_n,\lambda_n)$, for some $n\geq1$. Assume that there exists a $1$-cochain $(\phi,\psi)\in C^1_{tot}$ whose coboundary is the $n$-infinitesimal,
$$i.e., \delta_{tot}(\phi,\psi)=(\alpha_n,\mu_n,\lambda_n).$$
This gives
\begin{equation}\label{D12}
\alpha_n=\delta_H\phi,~\lambda_n=\delta_L\psi~\mbox{and}~\mu_n=-\delta_L\phi+\delta_v\psi.
\end{equation}
Take  $$\Phi_t=id_A+\phi t^n~~\mbox{and}~~\Psi_t=id_L+\psi t^n.$$ 
We define a deformation  $(\tilde{\alpha}_t,\tilde{\mu}_t,\tilde{\lambda}_t)$ of the Courant pair $((\alpha,\mu,\lambda)) $, where
$$\tilde{\alpha}_t=\Phi_t \circ \alpha_t \circ \Phi_t^{-1},~~ \tilde{\mu}_t(\Psi_t(x),~\Phi_t (a) )=\Phi_t \circ \mu_t (x,a),~~ \tilde{\lambda}_t = \Psi_t \circ \lambda_t  \circ \Psi_t^{-1}.$$
So, we have following equations:
\begin{align}
\tilde{\alpha_t}(\Phi_ta,\Phi_tb)&=\Phi_t\alpha_t(a,b),  \label{D13}\\
\tilde{\mu_t}(\Psi_tx,\Phi_ta)&=\Phi_t\mu_t(x,a),\label{D14}\\
\tilde{\lambda_t}(\Psi_tx,\Psi_ty)&=\Psi_t\lambda_t(x,y) .\label{D15}
\end{align}
By using above equations \eqref{D12}-\eqref{D15}, we have $$\tilde{\alpha}_n=\alpha_n-\delta_H\phi=0,~~\tilde{\lambda}_n=\lambda_n-\delta_L\psi=0,~~\tilde{\mu}_n=\mu_n+\delta_L\phi-\delta_v\psi=0.$$
and,$$\tilde{\alpha}_i=\tilde{\lambda}_i=\tilde{\mu}_i=0~~~\mbox{for}~1\leq i\leq n-1. $$
So, the given deformation $(\alpha_t,\mu_t,\lambda_t)$ is equivalent to a deformation $(\tilde{\alpha}_t,\tilde{\mu}_t,\tilde{\lambda}_t)$ for which $(\tilde{\alpha}_i,\tilde{\mu}_i,\tilde{\lambda}_i)=0,~\mbox{for}~1\leq i\leq n.$ Hence, we can repeat the argument to kill off any infinitesimal that is a coboundary. So the process must stop if the deformation is non trivial.
\end{proof}
\begin{Cor}
If $HL^2(A,L;A,L)=0$, then the Courant pair $(\alpha,\mu,\lambda) $ is rigid.\qed

\end{Cor}

\section{Obstruction Cocycles:}
 In this section, we discuss about the problem of realising a $2$-cocycle in $C_{tot}^2$ as the infinitesimal of a deformation of the Courant pair. To this end first we define an obstruction cochain and then proceed by detecting any obstructions to extend a given deformation modulo $t^n$ to a deformation modulo $t^{n+1},~n\geq1.$ Let $N$ be a positive integer.

\begin{Def}
A deformation of Courant pair of order $N$ is a $(\alpha_t,\mu_t,\lambda_t)$ such that $\alpha_t=\sum_{i=0}^N \alpha_i t^i$ and $\lambda_t=\sum_{i=0}^N \lambda_i t^i$ are deformations modulo $t^{n+1}$ of $A$ and $L$ respectively, that is $\alpha_t$ and $\lambda_t $ satisfy \eqref{D8}, \eqref{D7} respectively for $0\leq i\leq N,$ and $\mu_t=\sum_{i=0}^N \mu_i t^i$ satisfies equations \eqref{D5} and \eqref{D6}\\
$~~~~~$If there exists a $2$-cochain $(\alpha_{N+1},\mu_{N+1},\lambda_{N+1})\in C^2_{tot}$, such that the triple $(\tilde{\alpha_t},\tilde{\mu_t},\tilde{\lambda_t})$ with $\tilde{\mu_t}=\mu_t+\mu_{N+1}t^{N+1}$, $\tilde{\alpha_t}=\alpha_t+\alpha_{N+1}t^{N+1}$, and$\tilde{\lambda_t}=\lambda_t+\lambda_{N+1}t^{N+1}$ is a deformation of Courant pair of order $(N+1)$, then we say that $(\alpha_t,\mu_t,\lambda_t)$ extends to a deformation of Courant pair of order $(N+1)$.
\end{Def}
\begin{Def}\label{obstruction cochains}
 Let $(\alpha_t,\mu_t,\lambda_t)$ be a deformation of Courant pair of order $N$. Consider the cochains $\Theta(A),~\Theta^1~,\Theta^2,~\Theta(L)$, where 
\begin{align*}
\Theta(A)(a,b,c)&=\sum_{i+j=N+1;~ i,j>0}{\alpha_i(a,\alpha_j(b,c))-\alpha_i(\alpha_j(a,b),c)};\\
\Theta^1(x,a,b)&= \sum_{i+j=N+1;~ i,j>0}{\mu_i(x,\alpha_j(a,b))-\alpha_j(a,\mu_i(x,b))-\alpha_j(\mu_i(x,a),b)};\\
\Theta^2(x,y,a)&= \sum_{i+j=N+1;~ i,j>0}-{\mu_i(\lambda_j(x,y),a)+\mu_i(x,\mu_j(y,a))-\mu_i(y,\mu_j(x,a))};\\
\Theta(L)(x,y,z)&= \sum_{i+j=N+1;~ i,j>0}-{\lambda_i(x,\lambda_j(y,z))+\lambda_i(y,\lambda_j(x,z))+\lambda_i(\lambda_j(x,y),z)}.
\end{align*}
The $3$-cochain $\Theta(\alpha_t,\mu_t;\lambda_t)=(\Theta(A),~\Theta^1~,\Theta^2,~\Theta(L))\in C_{tot}^3$ is called the obstruction cochain for extending the deformation of Courant pair $(\alpha_t,\mu_t,\lambda_t)$ of order $N$ to a deformation of order $N+1$.
\end{Def}

\begin{Thm}\label{3cocycle}
The obstruction cochain $\Theta(\alpha_t,\mu_t,\lambda_t) = (\Theta(A),~\Theta^1~,\Theta^2,~\Theta(L))$ of a deformation of order $N$ is a $3$-cocycle in $C_{tot}^3$.
\end{Thm}
\begin{proof}
$\Theta(A)$ is the obstruction cochain for extending the $N$-th order deformation of the associative algebra $A$ to a deformation of order $N+1$,  So, it is a $3$-cocycle in the Hochschild complex for associative algebra, i.e.
\begin{equation}\label{D20}
\delta_H\Theta(A) = 0.
\end{equation}
 
Now, $\Theta(L)$ is the obstruction cochain for extending the $N$-th order deformation of  the Leibniz algebra $L$ to a deformation of order $N+1$. So, it is a $3$-cocycle in the deformation complex for a Leibniz algebra, i.e.  
\begin{equation}\label{D21}
\delta_L\Theta(L) = 0.
\end{equation}
We have the following equations using Lemma \ref{A.1}, \ref{A.2}, \ref{A.3}[see Appendix for details] :
\begin{align}
\delta_H\Theta^1-\delta_L\Theta(A)&=0;\label{D22}\\
\delta_H\Theta^2+\delta_L\Theta^1&=0;\label{D23}\\
\delta_v\Theta(L)-\delta_L\Theta^2&=0.\label{D24}
\end{align}
Using equations \eqref{D20}-\eqref{D24}, we have $\delta_{tot}(\Theta(\alpha_t,\mu_t,\lambda_t))=0$.
\end{proof}

\begin{Thm}
Let $(\alpha_t,\mu_t,\lambda_t)$ be a deformation of order $N$. Then $(\alpha_t,\mu_t,\lambda_t)$ extends to a deformation of order $N+1$ if and only if the cohomology class of the $3$-cocycle $\Theta(\alpha_t,\mu_t,\lambda_t)$ vanishes.
\end{Thm}
\begin{proof}
Suppose that a deformation $(\alpha_t,\mu_t,\lambda_t)$ of order $N$ extends to a deformation of order $N+1$. Then equations \eqref{D5}-\eqref{D8} hold for $n=N+1$. As a result, we get
$\theta(L)=\delta_L \mu_{N+1}$, $\theta(A)=\delta_H \alpha_{N+1}$, $\theta^1= \delta_L \alpha_{N+1}+\delta_H \lambda_{N+1}$ and $\Theta^2=\delta_v \mu_{N+1}-\delta_L \lambda_{N+1} $. In other words, the obstruction cochain $\theta (\alpha_t,\mu_t,\lambda_t)=d_{tot}(\alpha_{N+1},\lambda_{N+1},\mu_{N+1})$. So its cohomology class vanishes.\\

Conversely, let $\theta (\alpha_t,\mu_t,\lambda_t)$ be a coboundary. Suppose
$\theta (\alpha_t,\mu_t,\lambda_t)=d_{tot} (\alpha_{N+1},\lambda_{N+1},\mu_{N+1})$ for some $2$-cochain $(\alpha_{N+1},\lambda_{N+1},\mu_{N+1}) \in C^2_{tot}$.
Set $$(\tilde{\alpha}_t,\tilde{\lambda}_t,\tilde{\mu}_t)=(\alpha_t+\alpha_{N+1}t^{N+1},\lambda_t+\lambda_{N+1}t^{N+1},\mu_t+\mu_{N+1}t^{N+1}).$$ Then $(\tilde{\alpha}_t,\tilde{\lambda}_t,\tilde{\mu}_t) $ satisfies  equations \eqref{D5}-\eqref{D8} for $0\leq n\leq N+1$. So  $(\tilde{\alpha}_t,\tilde{\lambda}_t,\tilde{\mu}_t)$ is an extension of $(\alpha_t,\mu_t,\lambda_t)$ of order $N+1$.
\end{proof}
\begin{Cor}
If $HL^3(A,L;A,L)=0$, then every $2$-cocycle in $C^2_{tot}$ is the infinitesimal of some deformation.\qed
\end{Cor}
\begin{Exm}
Let L be the real three-dimensional Hisenberg Lie algebra with basis $\{e_1,e_2,e_3\}$. Then the Lie bracket $[-,-]: L\otimes L \rightarrow L$ is given by $[e_1,e_3]=-[e_3,e_1]=e_2$, and all other products of basis elements to be zero.

Define a linear map $\mu : L \rightarrow \chi (\mathbb{R}^3),$ where $\mu(e_1)=\frac \partial{\partial{x_1}}, \mu(e_2)=0, \mu(e_3)=0$. Then by considering the algebra of smooth functions  $C^{\infty}(\mathbb{R}^3)= A$, the Lie algebra $L$ equipped with the Lie algebra  homomorphism $\mu$ we  have a Leibniz pair $(A, L)= (C^{\infty}(\mathbb{R}^3), L)$.

Now by considering the Lie algebra $L$ as a Leibniz algebra and the map $\mu$ as morphism of Leibniz algebras we may treat  $(C^{\infty}(\mathbb{R}^3), L)$ as a Courant pair. It follows that  we have extra cohomology classes in the second cohomology space  $HL^2(A,L;A,L)$ of the Courant pair $(A,L)$ in comparison to the second cohomology space $H^2_{LP}(A,L;A,L)$ of the Leibniz pair $(A, L)$. 

Recall that $HL^2(L;L)$ denotes the second cohomology space of $L$ with coefficients in $L$ (considering $L$ as Leibniz algebra) and $H^2(L;L)$ denotes the second cohomology space of $L$ with coefficients in $L$ (considering $L$ as Lie algebra). Then the set $\{ [\phi_1],[\phi_2],[\phi_3]\}$, where $\phi_1(e_1,e_1)=e_2;~ \phi_2(e_3,e_3)=e_2; ~\phi_3(e_3,e_1)=e_2; ~\mbox{for all other}~1\leq i,~j \leq 3,~~\phi_1(e_i,e_j)=\phi_2(e_i,e_j)=\phi_3(e_i,e_j)=0,$ is a subset of a  basis of $HL^2(L;L)$, but none of these representing cocycle contained in the Lie algebra cohomology space $H^2(L;L)$. This computation is considered in the context of versal deformations of Leibniz algebras in \cite{Vers}.

  Applying the coboundary map $\delta_v$ in the deformation complex of Courant pair we have  $\delta_v(\phi_1)=\delta_v(\phi_2)=\delta_v(\phi_3)=0$. Consequently  
 $\{(0,0,[\phi_1]),(0,0,[\phi_2]),(0,0,[\phi_3])\}$ generates elements in the cohomology space $HL^2(A,L;A,L)$. But none of these elements  is in the Leibniz pair cohomology space $H_{LP}^2(A,L; A,L)$. 
 
 Define $(\alpha_t,\mu_t,\lambda_t)_i = (\alpha,\mu,\lambda)+t ~(0,0,\phi_i )$  for $i=1,2,3$. 
 Then we can check that these are  infinitesimal deformations of $(C^{\infty}(\mathbb{R}^3), L)$  which are also non-equivalent.
 
Moreover,  these deformations are obtained only when we consider the Leibniz pair $(C^{\infty}(\mathbb{R}^3), L)$ as a Courant pair. 

\end{Exm}

\newpage 

\appendix
\section{Obstruction Cocycle:} \label{App:Appendix}
In this appendix, we present an explicit computation involved  (in the proof of Theorem \ref{3cocycle} ) to show that an obstruction cochain arising in the extension of finite order deformation is  a $3$-cocycle. Here we follow the notations uded in the previous in the sections.

Let $(\alpha_t,\mu_t,\lambda_t)$ be a deformation of Courant pair $(A, L)$ of order $N$. Then for $n < N+1$, by using equations (\ref{D5}) to (\ref{D8}) respectively, we have following equations :
\begin{align}
 \sum_{i+j=n}\mu_i(x,\alpha_j(a,b))&=\sum_{i+j=n}{\alpha_i(a,\mu_j(x,b))+\alpha_i(\mu_j(x,a),b)}~~~ ~\mbox{for}~ x\in L,~ a,b \in A;\label{A25}\\
   \sum_{i+j=n}\mu_i(\lambda_j(x,y),a)&=\sum_{i+j=n}{\mu_i(x,\mu_j(y,a))-\mu_i(y,\mu_j(x,a))}~~~~ \mbox{for}~ x,y\in L,~ a\in A;\label{A26}\\
   \sum_{i+j=n}\lambda_i(x,\lambda_j(y,z))&=\sum_{i+j=n}{\lambda_i(\lambda_j(x,y),z)+\lambda_i(y,\lambda_j(x,z))} ~~~~\mbox{for}~ x,y,z\in L;\label{A27}\\
   \sum_{i+j=n}\alpha_i(a,\alpha_j(b,c))&=\sum_{i+j=n}\alpha_i(\alpha_j(a,b),c)~~~~~~ \mbox{for} ~a,b,c\in A \label{A28}.
\end{align}
Let us recall the cochains $\Theta(A),~\Theta^1~,\Theta^2,~\Theta(L)$ in Definition \ref{obstruction cochains},  given by the following equations:
\begin{align}
\Theta(A)(a,b,c)&=\sum_{i+j=N+1;~ i,j>0}{\alpha_i(a,\alpha_j(b,c))-\alpha_i(\alpha_j(a,b),c)} ;\label{A29}\\
\Theta^1(x,a,b)&= \sum_{i+j=N+1;~ i,j>0}{\mu_i(x,\alpha_j(a,b))-\alpha_j(a,\mu_i(x,b))-\alpha_j(\mu_i(x,a),b)};\label{A30}\\
\Theta^2(x,y,a)&= \sum_{i+j=N+1;~ i,j>0}-{\mu_i(\lambda_j(x,y),a)+\mu_i(x,\mu_j(y,a))-\mu_i(y,\mu_j(x,a))};\label{A31}\\
\Theta(L)(x,y,z)&= \sum_{i+j=N+1;~ i,j>0}-{\lambda_i(x,\lambda_j(y,z))+\lambda_i(y,\lambda_j(x,z))+\lambda_i(\lambda_j(x,y),z)}.\label{A32}
\end{align}
For $x\in L$, consider $\mu_i(x,-)=f_i^x(-)\in C^1(A,A)$. If we use the Gerstenhaber bracket for Hochschild cochains  (defined in remark \ref{DGLA}), for each $n<N+1$, we can rewrite the equations \eqref{A25} , \eqref{A26}, \eqref{A30} and \eqref{A31} respectively as follows:
\begin{equation}\label{c1}
\sum_{i+j=n}[f_i^x,\alpha_j]=0,
\end{equation}
\begin{equation}\label{c2}
\sum_{i+j=n}f_i^{\lambda_j(x,y)}=\sum_{i+j=n}[f_i^x,f_j^y].
\end{equation}
\begin{equation}\label{c3}
\Theta^1(x)=\sum_{i+j=N+1,~i,j>0}[f_i^x,\alpha_j],
\end{equation}
\begin{equation}\label{c4}
\Theta^2(x,y)=\sum_{i+j=N+1,~i,j>0}[f_i^x,f_j^y]-f_i^{\lambda_j(x,y)}.
\end{equation}
Next, we deduce the required identities satisfied by the above cochains.
\begin{Lem}\label{A.1}
The cochains $\Theta(A) \in Hom (A^3, A)$ and $\Theta^1 \in Hom (L, C^2(A,A))$ satisfy the identity,
$\delta_H\Theta^1-\delta_L\Theta(A)=0.$
\end{Lem}
\begin{proof} We want to show that for $x \in L$ and $a,b,c \in A$ we have
\begin{equation*}
\delta_H\Theta^1(x)(a,b,c)-\delta_L\Theta(A)(x)(a,b,c)=0.
\end{equation*}
From equation \eqref{c1}, for each $n< N+1$ we write as
\begin{equation*}
[\alpha_n,f_0^x]+[\alpha_0,f_n^x]=\sum_{i+j=n,~i,j>0}[f_i^x,\alpha_j].
\end{equation*}
By using the DGLA structure on the Hochschild cochain space ( see Remark \ref{DGLA} ) 
\begin{equation*}
\delta_L(\alpha_n)(x)+\delta_H(f_n^x)=\sum_{i+j=n,~i,j>0}[f_i^x,\alpha_j].
%\mbox{~~(using Gerstenhaber bracket and Remark \ref{DGLA})}
\end{equation*}
Further, $\alpha_t$ is a deformation of $\alpha$ so we have the following equation using obstruction cocycles for associative algebras. 
\begin{equation}\label{q1}
\Theta(A)=1/2\sum_{i+j=N+1,~i,j>0}[\alpha_i,\alpha_j].
\end{equation}
Now, for $x \in L$, $\delta_L\Theta(A)(x)=-[x,\Theta(A)]$ and for $a,b,c \in A$,
\begin{align*}
[x,\Theta(A)](a,b,c)
&=\mu(x,\Theta(A)(a,b,c))-\Theta(A)(\mu(x,a),b,c)-\Theta(A)(a,\mu(x,b),c)-\Theta(A)(a,b,\mu(x,c))\\
\vspace{0.1cm}
&=f_0^x~o~\Theta(A)(a,b,c)-\Theta(A)~o~\mu(a,b,c)\\
\vspace{0.1cm}&=[f_0^x,\Theta(A)](a,b,c).
\end{align*}
By using equation \eqref{q1}, we have
\begin{align*}
[x,\Theta(A)]&=1/2\sum_{i+j=N+1,~i,j>0}[f_0^x,[\alpha_i,\alpha_j]]\\
\vspace{0.1cm}
&=-\sum_{i+j=N+1,~i,j>0}[\alpha_i,\delta_L(\alpha_j)(x)],~~\mbox{where}~\delta_L(\alpha_j)(x)=[\alpha_j,f_0^x].
\end{align*}
Therefore,
\begin{equation}\label{A2}
\delta_L\Theta(A)(x)=-[x,\Theta(A)]=\sum_{i+j=N+1,~i,j>0}[\alpha_i,\delta_L(\alpha_j)]
\end{equation}
Also using equation \eqref{c3}, 
\begin{equation*}
\Theta^1(x)=\sum_{i+j=N+1,~i,j>0}[f_i^x,\alpha_j].
\end{equation*}
Using Remark \ref{DGLA}, $\delta_H(f)=(-1)^{|f|}[\alpha_0,f]$ for $f\in C^p$.  So, we have 
\begin{equation}\label{A1}
\begin{split}
&\delta_H\Theta^1(x)\\
&=-\sum_{i+j=N+1,~i,j>0}[\alpha_0,[f_i^x,\alpha_j]]\\
&=-\sum_{i+j=N+1,~i,j>0}[\alpha_j,[\alpha_0,f_i^x]]-\sum_{i+j=N+1,~i,j>0}[f_i^x,[\alpha_j,\alpha_0]]\\
%&=-\sum_{i+j=N+1,~i,j>0}[\alpha_j,-\delta_L(\alpha_i)(x)+\sum_{m+n=i}[f^x_m,\alpha_n]]-\sum_{i+j=N+1,~i,j>0}[f_i^x,-1/2\sum_{m+n=j,~m,n>0}[\alpha_m,\alpha_n]]\\
%&=\sum_{i+j=N+1,~i,j>0}[\alpha_j,\delta_L(\alpha_i)(x)]-\sum_{i+j+k=N+1,~i,j,k>0}[\alpha_j,[f^x_i,\alpha_k]]+1/2\sum_{i+j+k=N+1,~i,j,k>0}[f_i^x,[\alpha_j,\alpha_k]].\\
&=\sum_{i+j=N+1,~i,j>0}[\alpha_j,\delta_L(\alpha_i)(x)]- \frac{1}{2}\sum_{i+j+k=N+1,~i,j,k>0}[\alpha_k,[f^x_i,\alpha_j]]\\
&- \frac{1}{2}\sum_{i+j+k=N+1,~i,j,k>0}[\alpha_j,[f^x_i,\alpha_k]]
~~+1/2\sum_{i+j+k=N+1,~i,j,k>0}[f_i^x,[\alpha_j,\alpha_k]].
%~~~~~~\mbox{(using graded Jacobi identity)}
\end{split}
\end{equation}
Finally, using expressions in equations \eqref{A1}, \eqref{A2}, and graded Jacobi identity for the Gerstenhaber bracket (see Remark \ref{DGLA}), we have the required identity
$$\delta_H\Theta^1-\delta_L\Theta(A)=0.$$
\end{proof}

\begin{Lem}\label{A.2}
The cochains $\Theta^1 \in Hom (L,C^2(A, A))$ and $\Theta^2 \in Hom (L^2, C^1(A,A))$ satisfy the identity,
$\delta_H\Theta^2+\delta_L\Theta^1=0.$

\end{Lem}

\begin{proof}
Here, we will show that 
\begin{equation}\label{p1}
\delta_H\Theta^2(x,y)(a,b)+\delta_L\Theta^1(x,y)(a,b)=0
\end{equation}
for $x,y\in L$ and $a,b\in A$. From equation \eqref{c2}, for each $n<N+1$,
%$$\sum_{i+j=n}f_i^{\lambda_j(x,y)}=\sum_{i+j=n}[f_i^x,f_j^y].$$
\begin{equation}\label{A36}
f_0^{\lambda_n(x,y)}+f_n^{\lambda_0(x,y)}+\sum_{i+j=n,~i,j>0}f_i^{\lambda_j(x,y)}=\sum_{i+j=n,~i,j>0}[f_i^x,f_j^y]+[f_0^x,f_n^y]+[f_n^x,f_0^y].
\end{equation}
Now, 
\begin{align}
\delta_L\Theta^1(x,y)&=[x,\Theta^1(y)]-[y,\Theta^1(x)]-\Theta^1[x,y]\label{A37}\\
&=\sum_{i+j=N+1,~i,j>0}[f_0^x,[f_i^y,\alpha_j]]-[f_0^y,[f_i^x,\alpha_j]]-[f_i^{[x,y]},\alpha_j]\label{A38} ~~\mbox{(using equation \eqref{c3})}.
\end{align}
In equation \eqref{A38}, we can write the first two terms in the summation in right hand side as:
\begin{multline}\label{A43}
[f_0^x,[f_i^y,\alpha_j]]-[f_0^y,[f_i^x,\alpha_j]]= -[f_i^y,[f_j^x,\alpha_0]]+[f_i^x,[f_j^y,\alpha_0]]-\sum_{m+n=j,~m,n>0}[f_i^y,[f_m^x,\alpha_n]]\\
+\sum_{m+n=j,~m,n>0}[f_i^x,[f_m^y,\alpha_n]]-[\alpha_j,[f_0^x,f_i^y]]+[\alpha_j,[f_0^y,f_i^x]].
\end{multline}
Also in equation \eqref{A43}, last two terms of the right hand side can be written as,
\begin{equation}\label{A45}
\begin{split}
  -[\alpha_j,[f_0^x,f_i^y]]+[\alpha_j,[f_0^y,f_i^x]]&=[\alpha_j,-[f_0^x,f_i^y]+[f_0^y,f_i^x]]\\
  &=[\alpha_j,\sum_{m+n=i,~m,n>0}[f_m^x,f_n^y]-f_0^{\lambda_i(x,y)}-f_i^{[x,y]}-\sum_{m+n=i,~m,n>0}f_m^{\lambda_n(x,y)}].
  \end{split}
\end{equation}
Using equations \eqref{A45} and \eqref{A43}, we can write equation \eqref{A38} as follows:
\begin{multline}\label{A2.4}
\delta_L\Theta^1(x,y)=\sum_{i+j=N+1,~i,j>0}-[f_i^y,[f_j^x,\alpha_0]]+[f_i^x,[f_j^y,\alpha_0]]-\sum_{m+n=j,~m,n>0}[f_i^y,[f_m^x,\alpha_n]]\\
+[f_i^x,[f_m^y,\alpha_n]]+[\alpha_j,\sum_{m+n=i,~m,n>0}[f_m^x,f_n^y]-f_0^{\lambda_i(x,y)}-f_i^{[x,y]}-f_m^{\lambda_n(x,y)}].
\end{multline}
Now, we proceed to compute the expression of $\delta_H\Theta^2(x,y)$. First, note that using Remark \ref{DGLA} we have
\begin{equation*}
\delta_H\Theta^2(x,y)=[\alpha_0,\Theta^2(x,y)].
\end{equation*}
Consequently, by equation \eqref{c4}
%\begin{equation*}
%\Theta^2(x,y)=-\sum_{i+j=N+1,~i,j>0}f_i^{\lambda_j(x,y)}-\sum_{i+j=N+1,~i,j>0}[f_i^y,f_j^x]
%\end{equation*}
%So,
\begin{equation}
\delta_H\Theta^2(x,y)=-\sum_{i+j=N+1,~i,j>0}[\alpha_0,f_i^{\lambda_j(x,y)}]-\sum_{i+j=N+1,~i,j>0}[\alpha_0,[f_i^y,f_j^x]].\label{A42}
\end{equation}
Also, for fixed $j$, using Gerstenhaber bracket in equation \eqref{A25} we get 
\begin{equation}\label{A46}
  [\alpha_0,f_i^{\lambda_j(x,y)}]=\sum_{m+n=i,~m,n>0}[f_m^{\lambda_j(x,y)},\alpha_n]-[\alpha_i,f_0^{\lambda_j(x,y)}].
\end{equation}
If we consider equations \eqref{A46} and replace the first term on the right hand side in equation \eqref{A42}, we get the expression of $\delta_H\Theta^2(x,y)$.
\begin{equation}\label{A2.3}
\delta_H\Theta^2(x,y)=-\sum_{i+j=N+1,~i,j>0}~~\sum_{m+n=i,~m,n>0}[f_m^{\lambda_j(x,y)},\alpha_n]-[\alpha_i,f_0^{\lambda_j(x,y)}]-[\alpha_0,[f_i^y,f_j^x]]
\end{equation}
Finally, adding equations \eqref{A2.3} and \eqref{A2.4} and then applying graded Jacobi identity for the Gerstenhaber bracket, we get the required result
$$\delta_H\Theta^2+\delta_L\Theta^1=0.$$
\end{proof}
\begin{Lem}\label{A.3}
The cochains $\Theta^2 \in Hom (L^2, C^1(A,A))$ and $\Theta(L) \in Hom (L^3,L))$satisfy the identity,
$\delta_v\Theta(L)-\delta_L\Theta^2=0.$
\end{Lem}
\begin{proof}We want to show that
\begin{equation}\label{r1}
\delta_v\Theta(L)(x,y,z)(a)-\delta_L\Theta^2(x,y,z)(a)=0.
\end{equation}
for $x,y,z\in L$ and $a\in A$. 
First, we compute $\delta_L\Theta^2(x,y,z)(a)$. From definition  of $\delta_L$ we have
\begin{equation}\label{A47}
\begin{split}
&\delta_L\Theta^2(x,y,z)\\
=&-[x,\Theta^2(y,z)]+[y,\Theta^2(x,z)]-[z,\Theta^2(x,y)]
+\Theta^2([x,y],z)-\Theta^2(x,[y,z])+\Theta^2(y,[x,z]).
\end{split}
\end{equation}
%Now we proceed by considering expression of $\Theta^2(x,y)$ in equation \eqref{c4},
%\begin{equation*}
%\Theta^2(x,y)=\sum_{i+j=N+1,~i,j>0}[f_i^x,f_j^y]-f_i^{\lambda_j(x,y)}.
%\end{equation*}
In equation \eqref{A47}, using the  action of $L$ on $C^p(A,A)$ (recall from Section 4) and the Gerstenhaber bracket for Hochschild cochains, the first term in the right hand side can be written as:
\begin{equation*}\label{A49}
\begin{split}
[x,\Theta^2(y,z)]=&[f_0^x,\Theta^2(y,z)]\\
=& \sum_{i+j=N+1,~i,j>0}[f_0^x,[f_i^y,f_j^z]]-\sum_{i+j=N+1,~i,j>0}[f_0^x,f_i^{\lambda_j(y,z)}]].
\end{split}
\end{equation*}
Similarly, we write the expressions for $[y,\Theta^2(x,z)]$ and $[z,\Theta^2(x,y)]$ respectively.
%have
% \begin{equation*}
 %[y,\Theta^2(x,z)]=\sum_{i+j=N+1,~i,j>0}[f_0^y,[f_i^x,f_j^z]]-\sum_{i+j=N+1,~i,j>0}[f_0^y,f_i^{\lambda_j(x,z)}]];
 %\end{equation*}
%and
 %\begin{equation*}
 %[z,\Theta^2(x,y)]=\sum_{i+j=N+1,~i,j>0}[f_0^z,[f_i^x,f_j^y]]-\sum_{i+j=N+1,~i,j>0}[f_0^z,f_i^{\lambda_j(x,y)}]].
 %\end{equation*}
Further, in equation \eqref{A47}, the last three terms in the right hand side can be expressed as follows.
\begin{equation*}
  \Theta^2(x,[y,z])=\sum_{i+j=N+1,~i,j>0}[f_i^x,f_j^{\lambda_0(y,z)}]-f_i^{\lambda_j(x,\lambda_0(y,z))};
\end{equation*}

\begin{equation*}
  \Theta^2(y,[x,z])=\sum_{i+j=N+1,~i,j>0}[f_i^y,f_j^{\lambda_0(x,z)}]-f_i^{\lambda_j(y,\lambda_0(x,z))};
\end{equation*}

\begin{equation*}
  \Theta^2([x,y],z)=\sum_{i+j=N+1,~i,j>0}[f_i^{\lambda_0(x,y)},f_j^z]-f_i^{\lambda_j(\lambda_0(x,y),z)}.
\end{equation*}

Now, we will compute $\delta_v\Theta(L)(x,y,z)$ in equation \eqref{r1} as follows:
\begin{equation*}
\delta_v(\Theta(L)(x,y,z))(a)=f_0^{\Theta(L)(x,y,z)}(a).
\end{equation*}
By equation \eqref{A32},
\begin{equation*}
\Theta(L)(x,y,z)= \sum_{i+j=N+1;~ i,j>0}{-\lambda_i(x,\lambda_j(y,z))+\lambda_i(y,\lambda_j(x,z))+\lambda_i(\lambda_j(x,y),z)}.
\end{equation*}
Thus
\begin{equation}\label{r2}
\delta_v(\Theta(L)(x,y,z))(a)=\sum_{i+j=N+1;~ i,j>0}-f_0^{\lambda_i(x,\lambda_j(y,z))}+f_0^{\lambda_i(y,\lambda_j(x,z))}+f_0^{\lambda_i(\lambda_j(x,y),z)}(a).
\end{equation}
We proceed by rewritting the equation \eqref{A26} as follows:
\begin{equation}\label{A53}
   f_0^{\lambda_i(x,y)}=-f_i^{\lambda_0(x,y)}-\sum_{m+n=i,~m,n>0}f_m^{\lambda_n(x,y)}
   +\sum_{m+n=i,~m,n>0}[f_m^x,f_n^{y}]+[f_0^x,f_i^{y}]+[f_i^x,f_0^{y}].
\end{equation}
Replacing $y$ by $\lambda_j(y,z)$, we can write the first term in the right hand side of the equation \eqref{r2} as follows.
\begin{equation}\label{A54}
\begin{split}
   f_0^{\lambda_i(x,\lambda_j(y,z))}=&-f_i^{\lambda_0(x,\lambda_j(y,z))}-\sum_{m+n=i,~m,n>0}f_m^{\lambda_n(x,\lambda_j(y,z))}
   +\sum_{m+n= i,~m,n>0}[f_m^x,f_n^{\lambda_j(y,z)}]\\
   &+[f_0^x,f_i^{\lambda_j(y,z)}]+[f_i^x,f_0^{\lambda_j(y,z)}].
   \end{split}
\end{equation}
Using equation \eqref{A53}, the last term in the right hand side of equation \eqref{A54} can be written as,
\begin{multline}\label{A58}
  [f_i^x,f_0^{\lambda_j(y,z)}]=-[f_i^x,f_j^{\lambda_0(y,z)}]-\sum_{p+q=j,~p,q>0}[f_i^x,f_p^{\lambda_q(y,z)}]+\sum_{p+q=j,~p,q>0}[f_i^x,[f_p^y,f_q^z]]
                                \\+[f_i^x,[f_0^y,f_j^z]]+[f_i^x,[f_j^y,f_0^z]].
\end{multline}
Similarly, we can write the second and third term in the right hand side of the  equation \eqref{r2}. 

Now, substituting the final expression of $\delta_v\Theta(L)(x,y,z)(a)$ and $\delta_L\Theta^2(x,y,z)(a)$  in the left hand side of the equation (\ref{r1}) and
%is of the following form:
%\begin{multline}\label{r3}
%\Bigg(\sum_{i+j=N+1,~i,j>0}f_i^{\lambda_j(x,\lambda_0(y,z))}+\sum_{i+j=N+1,~i,j>0}f_i^{\lambda_0(x,\lambda_j(y,z))}
%+\sum_{i+j+k=N+1,~i,j,k>0}f_i^{\lambda_j(x,\lambda_k(y,z))}\\-\sum_{i+j=N+1,~i,j>0}f_i^{\lambda_j(\lambda_0(x,y),z)}-\sum_{i+j=N+1,~i,j>0}f_i^{\lambda_0(\lambda_j(x,y),z)}
%-\sum_{i+j+k=N+1,~i,j,k>0}f_i^{\lambda_j(\lambda_k(x,y),z)}\\
%-\sum_{i+j=N+1,~i,j>0}f_i^{\lambda_j(y,\lambda_0(x,z))}-\sum_{i+j=N+1,~i,j>0}f_i^{\lambda_0(x,\lambda_j(y,z))}
%-\sum_{i+j+k=N+1,~i,j,k>0}f_i^{\lambda_j(y,\lambda_k(x,z))}\Bigg)(a).
%\end{multline}
%Finally the result follows
 by considering equation \eqref{A27} we get the result.
 %and using the expression we get for each $n<N+1$
%\begin{equation}\label{A59}
 % \sum_{i+j=n}{-\lambda_i(x,\lambda_j(y,z))+\lambda_i(y,\lambda_j(x,z))+\lambda_i(\lambda_j(x,y),z)}=0.
%\end{equation}
\end{proof}

{\bf Conclusions:}

The deformation of a given structure characterize the local behaviour in the variety of a given type of objects. In order to study the deformation one needs a suitable notion of module or representation  and then construct a deformation complex. In this work we develop a formal deformation of Courant pairs which includes the classical cases of Lie algebras and  Leibniz pairs. Moreover, There are plenty of examples of the Courant pairs including the Leibniz algebras and Leibniz algebroids  appear in algebra, geometry and in Mathematical physics. A Courant pair $(A,L)$ is a Leibniz pseudoalgebra over $(\mathbb{K}, A )$  if $L$ is also an $A$-module  and the structure map is an $A$-module homomorphism such that following condition is satisfied:
$[x,fy]=f[x,y]+\mu(x)(f)y$ where $x,y\in L$ and $f\in A$. So it is also natural to expect a deformation theory for Leibniz- Rinehart algebras by deducing the required deformation cohomology.  As in the classical cases of algebras over quadratic operads (e.g., Ass, Lie, Leib,...etc. ), one may look for a differential graded Lie algebra structure on the cohomology space of a Courant pair. Also, a Courant pair can be described as a non skew-symmetric version of OCHA and subsequently as an algebra over an operad. These questions are planed to be addressed  in a separate discussion.

%\newpage

%\end{document}
%\vspace{.5cm}
{\bf Ashis Mandal}\\
 Department of Mathematics and Statistics,
Indian Institute of Technology,
Kanpur 208016, India.\\
e-mail: amandal@iitk.ac.in

\vspace{.5cm}
{\bf Satyendra Kumar Mishra}\\
Department of Mathematics and Statistics,
Indian Institute of Technology,
Kanpur 208016, India.\\
e-mail: satyendm@iitk.ac.in

\end{document}